\documentclass[12pt]{amsart}
\usepackage{amsmath}
\usepackage{amssymb}
\usepackage{amsfonts}
\usepackage{amsthm}
\usepackage{xcolor}
\usepackage{hyperref}
\hypersetup{
    colorlinks=true,       
    linkcolor=blue,          
    citecolor=green,        
    filecolor=cyan,         
    urlcolor=magenta        
}

\newtheorem{theo}{Theorem}[section]
\newtheorem*{theo*}{Theorem}
\newtheorem{coro}[theo]{Corollary}

\title[]{On near orthogonality of certain $k$-vectors involving generalized Ramanujan sums}
\begin{document}
 
 \subjclass[2010]{}
 
 \author[N E Thomas]{Neha Elizabeth Thomas}
 \address{Department of Mathematics, University College, Thiruvananthapuram (Research Centre of the University of Kerala), Kerala - 695034, India}
 \email{nehathomas2009@gmail.com}

\author[K V Namboothiri]{K Vishnu Namboothiri}
\address{Department of Mathematics, Baby John Memorial Government College, Chavara, Sankaramangalam, Kollam, Kerala - 691583, INDIA\\Department of Collegiate Education, Government of Kerala, India}
\email{kvnamboothiri@gmail.com}
\thanks{}

 \begin{abstract}
The near orthgonality of certain $k$-vectors involving the Ramanujan sums were studied by E. Alkan in [J. Number Theory, 140:147--168 (2014)]. Here we undertake the study of similar vectors involving a generalization of the Ramanujan sums defined by E. Cohen in [Duke Math. J., 16(2):85--90 (1949)]. We also prove that the weighted average $\frac{1}{k^{s(r+1)}}\sum \limits_{j=1}^{k^s}j^rc_k^{(s)}(j)$ remains positve for all $r\geq 1$. Further, we give a lower bound for $\max\limits_{N}\left|\sum \limits_{j=1}^{N^s}c_k^{(s)}(j) \right|$.
 \end{abstract}
 
 \maketitle
\section{Introduction}
For positive integer $k$ and complex number $z$, Srinivasa Ramanujan introduced the sum
 \begin{align*}
 c_k(z):= \sum\limits_{\substack{m=1\\(m, k)=1}}^{k}e^{\frac{2\pi imz}{k}}.
 \end{align*}
in \cite{ramanujan1918certain}. He  obtained Fourier series like representations for many well-known arithmetical functions in terms of these sums. These sums were called as the \emph{Ramanujan sums} later. Some orthogonality properties of the Ramanujan sums were derived by Carmichael in \cite{carmichael1932expansions}. He showed that

\begin{align*}
 \sum\limits_{j=1}^Nc_{k_1}(j)c_{k_2}(j)=0
\end{align*}
whenever $k_1 \neq k_2$ and $k_1, k_2$ both divide $N$, and
\begin{align*}
 \sum\limits_{j=1}^Nc_{k}(j)^2=\phi(k)N
\end{align*} whenever $k|N$. Here  $\phi$ denotes the Euler totient function.
Using these orthogonality properties, Carmichael established the existence of Fourier series like representations for certain arithmetical functions in \cite{carmichael1932expansions}.

In \cite{alkan2014ramanujan}, E. Alkan considered a family of $\mathbb{R}^k$ vectors (where $\mathbb{R}$ is the set of all real numbers and $k$ a positive integer) with components involving the Ramanujan sums. In this family, he proved that though orthogonality could not be achieved, they can found to be very close to being orthogonal. Precisely, in terms of the usual inner product on $\mathbb{R}^k$, Alkan proved that for many values of $k$ and $r$, the inner product

\begin{align*}
 \frac{1}{k^{r+1}}\langle 1^r,2^r\ldots,k^r \rangle .\langle c_k(1),c_k(2),\ldots, c_k(k)\rangle =\frac{1}{k^{r+1}}\sum \limits_{j=1}^{k}j^rc_k(j)
\end{align*}
is positive and very close to zero. Alkan also proved that
\begin{align*}
 \max\limits_N \left|\sum\limits_{j=1}^Nc_{k}(j)\right|\geq \frac{J_{2}(k)}{4k}+\frac{\phi(k)}{2}
\end{align*} where $J_s(k)$ denotes the Jordan totient function. In addition, he discussed various features of the sum $\sum \limits_{j=1}^{k}j^rc_k(j)$ including that this sum is always positive irrespective of the values of $k$ and $r$. We would like to note that this weighted power sum  appeared in many of the problems discussed by Alkan in \cite{alkan2011mean}   \cite{alkan2013averages}, and \cite{alkan2014ramanujan}. This sum itself became a major point of discussion in many other papers. For example, Alkan himself derived a formula for $S_r(k) = \frac{1}{k^{r+1}}\sum \limits_{j=1}^{k}j^rc_k(j)$ in \cite{alkan2012distribution} and L. T\'{o}th derived another formula for $S_r(k)$ in \cite{toth2014averages}.

In \cite{cohen1949extension}, E. Cohen gave a generalization of the Ramanujan sum (called hereafter as the \emph{generalized Ramanujan sum} or the \emph{Cohen-Ramanujan sum})  defining
\begin{align*}
 c_k^{(s)}(j):=\sum \limits_{\substack{m = 1\\ (m, k^s)_s=1}}^{k^s}e^{\frac{2\pi ijm}{k^s}}
\end{align*}
where $(a, b)_s$ is the generalized GCD of $a$ and $b$  (see the definition in the next section).

A generalization of the sum $S_r(k)$ involving this generalization of the Ramanujan sums was discussed  by K V Namboothiri in \cite{namboothiri2017certain} and its asymptotic properties were studied by I. Kiuchi in \cite{kiuchi2017sums}.

We would also like to mention another work in which the near orthogonality was discussed. Asymptotic orthogonality of the M\"{o}bius function to a particular type of nilsequences was established by Green and Tao in \cite{green2012mobius}. This near orthogonality is closely related to \emph{generalized the Hardy-Littlewood conjecture}. Please see  \cite{green2012mobius} for a detailed discussion on this.

The main goal of this paper is to study some of the problems discussed by Alkan in \cite{alkan2014ramanujan} on the near orthogonality concepts of vectors involving generalized Ramanujan sums. We here prove that certain vectors in $\mathbb{R}^k$ formed using generalized Ramanujan sums are also nearly orthogonal. We also prove that the inner product
\begin{align*}
\frac{1}{k^{s(r+1)}}\langle 1^r,2^r\ldots,(k^s)^r \rangle .\langle c_k^{(s)}(1),c_k^{(s)}(2),\ldots, c_k^{(s)}(k^s)\rangle =\frac{1}{k^{s(r+1)}}\sum \limits_{j=1}^{k^s}j^rc_k^{(s)}(j)
\end{align*} is always positive irrespective of the values of $k,  r$, and $s$. Further, we establish that
\begin{align*}
\max\limits_{N}\left|\sum \limits_{j=1}^{N^s}c_k^{(s)}(j) \right|\geq \frac{J_{2s}(k)}{4k^s}+\frac{J_s(k)}{2}.
\end{align*}
E. Cohen has proved in a series of papers \cite{cohen1949extension, cohen1955extension, cohen1956extension} that many of the results involving Ramanujan sums have a natural generalization to problems involving the generalized Ramanujan sums. Analogously, our discussions in this paper show that the results of Alkan in \cite{alkan2012distribution} and \cite{alkan2014ramanujan} involving Ramanujan sums have such natural generalizations to analogous results involving generalized Ramanujan sums. For more problems related to the concepts discussed above, please see the papers \cite{alkan2022generalization}, \cite{ikeda2016sums}, \cite{ma2023hybrid}, and \cite{singh2014sums}.

\section{Notations and basic results}
In this section, we will introduce the most commonly used definitions and results required in the following sections. If any of the terms appearing in this paper are not defined in this section, they will be introduced just before their first use, or otherwise, their definitions can be found in \cite{tom1976introduction} or \cite{montgomery2006multiplicative}.

For any positive integer $k$ with $k>1$, we have
 \begin{align}\label{mu}
  \sum\limits_{d|k} \mu(d) = 0\text{ \cite[Theorem 2.1]{tom1976introduction}}.
 \end{align} where $\mu$ denotes the usual M\"{o}bius function.
 Analogous to the Euler totient function $\phi$,
the \textit{Jordan totient function} $J_k(n)$ is defined by
\begin{align}\label{J_s}
J_k(n) := n^k\prod_{\substack{p|n\\p\text{ prime}}}\left(1-\frac{1}{p^k}\right).
\end{align} 
From \cite[Chapter 2]{tom1976introduction}, we have
\begin{align}
 J_s(n)=\sum\limits_{d|n}\mu(d)\left(\frac{n}{d}\right)^s.
\end{align}

Now we recall the definition of the Bernoulli polynomials and Bernoulli numbers.
For any $x \in \mathbb{C}$, the \textit{Bernoulli polynomials} $B_n(x)$ is defined by the equation
\begin{align*}
 \frac{ze^{xz}}{e^z-1}= \sum\limits_{n=0}^{\infty}\frac{B_n(x)}{n!}z^n \text{ where } |z|<2\pi.
\end{align*}
The numbers $B_n(0)$ are called \textit{Bernoulli numbers} and are denoted by $B_n$. Thus
\begin{align*}
 \frac{z}{e^z-1}= \sum\limits_{n=0}^{\infty}\frac{B_n}{n!}z^n \text{ where } |z|<2\pi.
\end{align*}
Also by \cite[Theorem 12.14]{tom1976introduction}, we have $ B_n=B_n(0)=B_n(1)$. Further we have $B_n=0$ when $n$ is odd and $n\geq3$. Now from \cite[Theorem 12.15]{tom1976introduction} we get a basic recursion of Bernoulli numbers
as
\begin{align}\label{recursionformula}
 \sum\limits_{j=0}^{n}\binom{n+1}{j}B_j=0 \text{ for } n\geq1.
\end{align}
 Also from \cite[Theorem 12.12]{tom1976introduction} we have $B_n(x)=\sum \limits_{k=0}^n\binom{n}{k}B_kx^{n-k}$.
See \cite[Chapter 12]{tom1976introduction} for  other properties of Bernoulli numbers and Bernoulli polynomials.
 
 As introduced in the first section, the \textit{Ramanujan sum} $c_k(j)$ is defined as follows:
 \begin{align}
  c_k(j):= \sum\limits_{\substack{m=1\\(m, k)=1}}^{k}e^{\frac{2\pi ijm}{k}}.
 \end{align}

For $a,b$ positive integers, $(a, b)_s$ will denote the \emph{generalized GCD} of $a$ and $b$ defined to be the largest $d^s \in \mathbb{N}$ (where $d\in \mathbb{N}$) such that $d^s|a$ and $d^s|b$.  For $s$, $k \in \mathbb{N}$, $j\in \mathbb{Z}$, the \textit{generalized Ramanujan sum} (\cite{cohen1949extension})  is defined as
 \begin{align}
  c_k^{(s)}(j):=\sum \limits_{\substack{m = 1\\ (m, k^s)_s=1}}^{k^s}e^{\frac{2\pi ijm}{k^s}}.
 \end{align}
 From
\cite[Theorem 1]{cohen1956extension} we have
\begin{align}\label{generalizedJ}
c_k^{(s)}(j)=\frac{J_s(k)\mu(d)}{J_s(d)}
\end{align}
where $d^s=\frac{k^s}{(j, k^s)_s}$.

Let $A$ be a set of integers. The \textit{counting function} $A(x)$ of the set $A$ counts the number of positive elements of $A$ not exceeding $x$. That is,
\begin{align*}
 A(x):=\sum \limits_{\substack{a \in A \\ 1 \leq a \leq x}}1.
\end{align*}
The \textit{lower asymptotic density} of $A$ is defined by
\begin{align*}
 d_L(A):=\liminf \limits_{x\rightarrow \infty}\frac{A(x)}{x}.
\end{align*}
Similarily the \textit{upper asymptotic density} of $A$ is defined by
\begin{align*}
 d_U(A):=\limsup \limits_{x\rightarrow \infty}\frac{A(x)}{x}.
\end{align*}
  The set $A$ has \textit{asymptotic density} $d(A)=\alpha$ if $d_L(A)=d_U(A)=\alpha$ or equivalently, $\lim \limits_{x\rightarrow \infty}\frac{A(x)}{x} = \alpha$. To know more about concepts about density, please see \cite{nathanson2008elementary}.

 If $\mathbb{P}$ is the set of all prime numbers and if $P$ is any subset of $\mathbb{P}$, then the \textit{set of integers of Beurling type corresponding to $P$} is  the semigroup $\left<P\right>$ generated by all the primes in $P$, namely
 \begin{align*}
  \left<P\right> := \left\lbrace\prod \limits_{j=1}^{k}p_{j}^{a_j}:k \geq 1, a_j \geq 0, p_j \in P \right\rbrace
 \end{align*}
with the convention that $\left<P\right> = \{1\}$ when $P$ is empty.

Define $\omega(k)$ to be the function giving the number of distinct prime divisors of $k$.

Now we proceed to state our main results and prove them.
 \section{Main Results and proofs}
 Our first result is a generalization of  \cite[Theorem 2]{alkan2012distribution}.
 \begin{theo}\label{theorem 1}
 Let $B=\left<P\right>$ be a set of integers of Beurling type corresponding to a subset $P$ of all primes with 
 \begin{align*}
 N_P(x)=cx+O\left(\frac{x}{(logx)^\lambda}\right)
 \end{align*}
 where $c>0$ and $\lambda>0$ are constants and $N_P(x)$ is the counting function of $B$. Let $k$ and $r$ be positive integers such that  $k\geq2$ and $r\geq2$. Then, for any $\epsilon$ with $0<\epsilon<1$, there exists a subset $B_\epsilon \subseteq B$ depending only on $\epsilon$ and having positive density such that
 \begin{align*}
 \left|\frac{1}{k^{s(r+1)}}\sum \limits_{j=1}^{k^s}j^rc_k^{(s)}(j)-\left( \frac{J_s(k)}{2k^s}-\frac{1}{r+1}\right)\right|<\epsilon
\end{align*} 
for any $k\in B_\epsilon$.
 \end{theo}
 \begin{proof}
From \cite[Proposition 11]{namboothiri2017certain}, we have
\begin{align}\label{generalized}
\frac{1}{k^{s(r+1)}}\sum \limits_{j=1}^{k^s}j^rc_k^{(s)}(j)= \frac{J_s(k)}{2k^s}+\frac{1}{r+1}\sum \limits_{m=1}^{[\frac{r}{2}]}\binom{r+1}{2m}B_{2m}\frac{J_{2ms}(k)}{k^{2ms}}
\end{align}
for any natural number $r$ and $k$.
Using equation (\ref{J_s}), we may rewrite the above as
\begin{align}\label{weighted averagesB_2m/d^2ms}
\frac{1}{k^{s(r+1)}}\sum \limits_{j=1}^{k^s}j^rc_k^{(s)}(j)&= \frac{J_s(k)}{2k^s}+\frac{1}{r+1}\sum \limits_{m=1}^{[\frac{r}{2}]}\binom{r+1}{2m}B_{2m}\sum\limits_{d|k}\frac{\mu(d)}{d^{2ms}}\\
&=\frac{J_s(k)}{2k^s}+\frac{1}{r+1}\sum\limits_{d|k}\mu(d)\sum \limits_{m=1}^{[\frac{r}{2}]}\binom{r+1}{2m}\frac{B_{2m}}{d^{2ms}}. \nonumber
\end{align}
Let $r$ be an even positive integer.
By \cite[Equation 3.3]{alkan2012distribution}, we have
\begin{align*}
B_{r+1}(x)&=x^{r+1}-\frac{(r+1)}{2}x^r+x^{r+1}\sum\limits_{m=1}^{\frac{r}{2}}\binom{r+1}{2m}\frac{B_{2m}}{ x^{2m}}
\end{align*}
so that
\begin{align*}
B_{r+1}(d^s)=d^{s(r+1)}-\frac{(r+1)}{2}d^{sr}+d^{s(r+1)}\sum\limits_{m=1}^{\frac{r}{2}}\binom{r+1}{2m}\frac{B_{2m}}{ d^{2ms}}.
\end{align*}
and so
\begin{align}\label{B_2m/d^2ms}
\sum\limits_{m=1}^{\frac{r}{2}}\binom{r+1}{2m}\frac{B_{2m}}{ d^{2ms}}
=\frac{B_{r+1}(d^s)}{d^{s(r+1)}}-1+\frac{r+1}{2d^s}.
\end{align}
Hence
\begin{align*}
&\sum\limits_{d|k}\mu(d)\sum \limits_{m=1}^{[\frac{r}{2}]}\binom{r+1}{2m}\frac{B_{2m}}{d^{2ms}}\\
&=\sum\limits_{d|k}\mu(d)\left(\frac{B_{r+1}(d^s)}{d^{s(r+1)}}-1+\frac{r+1}{2d^s}\right)\quad(\text{from } (\ref{B_2m/d^2ms}))\\&=\sum\limits_{d|k}\mu(d)\frac{B_{r+1}(d^s)}{d^{s(r+1)}}-\sum\limits_{d|k}\mu(d)+\frac{r+1}{2}\sum\limits_{d|k}\frac{\mu(d)}{d^s}\\
&=\sum\limits_{d|k}\mu(d)\frac{B_{r+1}(d^s)}{d^{s(r+1)}}+(r+1)\frac{J_s(k)}{2k^s}\quad(\text{from (\ref{mu}) and (\ref{J_s})}).
\end{align*}
Therefore (\ref{weighted averagesB_2m/d^2ms}) becomes
\begin{align}\label{c_k^s(j)d>1}
\frac{1}{k^{s(r+1)}}\sum \limits_{j=1}^{k^s}j^rc_k^{(s)}(j)&= \frac{J_s(k)}{2k^s}+\frac{1}{r+1}\sum\limits_{d|k}\mu(d)\frac{B_{r+1}(d^s)}{d^{s(r+1)}}+\frac{J_s(k)}{2k^s} \\
&=\frac{J_s(k)}{k^s}+\frac{1}{r+1}\sum\limits_{\substack{d|k\\d>1}}\mu(d)\frac{B_{r+1}(d^s)}{d^{s(r+1)}}\quad\text{ (since }B_{r+1}(1)=0). \nonumber
\end{align}
For $n>1$, let $S_r(n):=\sum \limits_{j=1}^{n-1}j^r$ denote the sum of consecutive $r$\textsuperscript{th} powers.\\
We have $(r+1)S_r(n)=\sum \limits_{j=0}^r\binom{r+1}{j}B_jn^{r+1-j}=B_{r+1}(n)-B_{r+1}$ \cite[Theorem 1, Chapter 15]{ireland1990classical}. Since $r$ is even and $r\geq2$, we have $B_{r+1}=0$ and hence $S_r(n)=\frac{B_{r+1}(n)}{r+1}$. Therefore
 \begin{align}\label{S_r(d^s)}
S_r(d^s)=\frac{B_{r+1}(d^s)}{r+1}
\end{align} for any $d>1$.
Combining (\ref{c_k^s(j)d>1}) and (\ref{S_r(d^s)}), we get
\begin{align}\label{S_r(d^s)even}
\frac{1}{k^{s(r+1)}}\sum \limits_{j=1}^{k^s}j^rc_k^{(s)}(j)=\frac{J_s(k)}{k^s}+\sum\limits_{\substack{d|k\\d>1}}\mu(d)\frac{S_r(d^s)}{d^{s(r+1)}}
\end{align}
for any even $r$ where $r\geq 2$.
Now let $r$ be an odd positive integer.
From \cite[Equation 3.9]{alkan2012distribution}, we have
\begin{align*}
B_{r+1}(x)=x^{r+1}-\frac{(r+1)}{2}x^r+x^{r+1}\sum\limits_{m=1}^{\frac{r-1}{2}}\binom{r+1}{2m}\frac{B_{2m}}{ x^{2m}}+B_{r+1}.
\end{align*}
Therefore
\begin{align*}
B_{r+1}(d^s)=d^{s(r+1)}-\frac{(r+1)}{2}d^{sr}+d^{s(r+1)}\sum\limits_{m=1}^{\frac{r-1}{2}}\binom{r+1}{2m}\frac{B_{2m}}{ d^{2ms}}+B_{r+1}
\end{align*} 
so that
\begin{align*}
\sum\limits_{m=1}^{\left[\frac{r}{2}\right]}\binom{r+1}{2m}\frac{B_{2m}}{ d^{2ms}}&=\sum\limits_{m=1}^{\frac{r-1}{2}}\binom{r+1}{2m}\frac{B_{2m}}{ d^{2ms}}\\&= \frac{1}{d^{s(r+1)}}\left( B_{r+1}(d^s)-d^{s(r+1)}+\frac{(r+1)}{2}d^{sr}-B_{r+1}\right).
\end{align*}
Hence
\begin{align*}
\sum\limits_{d|k}\mu(d)\sum \limits_{m=1}^{[\frac{r}{2}]}\binom{r+1}{2m}\frac{B_{2m}}{d^{2ms}}
&=\sum\limits_{d|k}\mu(d)\left(\frac{B_{r+1}(d^s)}{d^{s(r+1)}}-1+\frac{(r+1)}{2d^s}-\frac{B_{r+1}}{d^{s(r+1)}}\right)\\
&=\sum\limits_{d|k}\frac{\mu(d)}{d^{s(r+1)}}\left(B_{r+1}(d^s)-B_{r+1}\right)+(r+1)\frac{J_s(k)}{2k^s}.
\end{align*}
Therefore from equation (\ref{weighted averagesB_2m/d^2ms}), we have
\begin{align}\label{c_k^s(j)odd}
&\frac{1}{k^{s(r+1)}}\sum \limits_{j=1}^{k^s}j^rc_k^{(s)}(j)\\&= \frac{J_s(k)}{2k^s}+\frac{1}{r+1}\sum\limits_{d|k}\frac{\mu(d)}{d^{s(r+1)}}\left(B_{r+1}(d^s)-B_{r+1}\right)+\frac{J_s(k)}{2k^s}\nonumber\\
&= \frac{J_s(k)}{k^s}+\frac{1}{r+1}\sum\limits_{\substack{d|k\\d>1}}\frac{\mu(d)}{d^{s(r+1)}}\left(B_{r+1}(d^s)-B_{r+1}\right)(\text{ since } B_{r+1}(1)=B_{r+1}).\nonumber
\end{align}
As noted previously, $S_r(d^s)=\frac{1}{r+1}\left(B_{r+1}(d^s)-B_{r+1}\right)$ for any positive integer $d$ with $d>1$ and so
\begin{align}\label{S_r(d^s)odd}
\frac{1}{k^{s(r+1)}}\sum \limits_{j=1}^{k^s}j^rc_k^{(s)}(j)=\frac{J_s(k)}{k^s}+\sum\limits_{\substack{d|k\\d>1}}\mu(d)\frac{S_r(d^s)}{d^{s(r+1)}}
\end{align} holds for odd $r$ as well. Now for $d>1$,
\begin{align}\label{L_d,r}
L_{d,r}=\frac{S_r(d^s)}{d^{s(r+1)}}=\frac{1}{d^{s(r+1)}}\sum\limits_{j=1}^{d^s-1}j^r=\frac{1}{d^s}\sum\limits_{j=1}^{d^s-1}\left(\frac{j}{d^s}\right)^r
\end{align} 
is a lower Riemann sum using the equally spaced partitioning points $\left\lbrace\frac{1}{d^s}, \frac{2}{d^s},\cdots, \frac{d^s-1}{d^s} \right\rbrace$ corresponding to the integral $\int\limits_0^1 x^r = \frac{1}{r+1}$. For any $d>1$, $L_{d,r}<\frac{1}{r+1}$ and thus we may define the numbers
\begin{align}\label{delta_d,r}
\delta_{d^s,r}:=\frac{1}{r+1}-\frac{S_r(d^s)}{d^{s(r+1)}}>0
\end{align}
as an error in the approximation of this integral by the lower Riemann integral. If we consider the upper Riemann sum for the same integral defined as
\begin{align}\label{U_d,r}
U_{d,r}:=\frac{1}{d^s}\sum\limits_{j=1}^{d^s}\left(\frac{j}{d^s}\right)^r,
\end{align}
we can see that $L_{d,r}<\frac{1}{r+1}<U_{d,r}$ and $U_{d,r}-L_{d,r}=\frac{1}{d^s}$. It follows from (\ref{L_d,r})-(\ref{U_d,r}) that $0<\delta_{d^s,r}<\frac{1}{d^s}$ for any $d>1$ and $r\geq2$. Now from (\ref{S_r(d^s)odd}) and (\ref{delta_d,r}), we can see that
\begin{align}\label{c_k^s(j)delta}
\frac{1}{k^{s(r+1)}}\sum \limits_{j=1}^{k^s}j^rc_k^{(s)}(j)
&=\frac{J_s(k)}{k^s}+\sum\limits_{\substack{d|k\\d>1}}\mu(d)\left(\frac{1}{r+1}-\delta_{d^s,r}\right)\\
&=\frac{J_s(k)}{k^s}-\frac{1}{r+1}-\sum\limits_{\substack{d|k\\d>1}}\mu(d)\delta_{d^s,r}(\text{ since } \sum\limits_{\substack{d|k\\d>1}}\mu(d)=-1) \nonumber
\end{align}
follows for any $k\geq2$. Let $q$ be a prime number and let $\mathbb{N}_q$ and $P_q$  be defined as in the proof of  \cite[Theorem 2]{alkan2012distribution}. Hence from \cite[Equations 3.18-3.25]{alkan2012distribution} there exists a   subset $B_\epsilon \subseteq B\cap \mathbb{N}_q \subseteq B$ depending only on $\epsilon$ and having positive density such that
\begin{align}\label{sum1/p<}
\sum \limits_{p|k}\frac{1}{p}\leq \frac{\epsilon}{2}
\end{align}
holds for all $k \in B_\epsilon$. Now
\begin{align}\label{mudelta}
\sum\limits_{\substack{d|k\\d>1}}\mu(d)\delta_{d^s,r}= \sum\limits_{p_1|k}\delta_{p_1^s,r}-\sum\limits_{p_1p_2|k}\delta_{(p_1p_2)^s,r}+\sum\limits_{p_1p_2p_3|k}\delta_{(p_1p_2p_3)^s,r}-\cdots
\end{align}
where $p_j$'s are the distinct prime divisors of $k$. There are only finitely many non-zero sums on the right side of (\ref{mudelta}). If $k\in B_\epsilon$, then using (\ref{sum1/p<}) and the fact that
$0<\delta_{d^s,r}<\frac{1}{d^s}$, we get
\begin{align}\label{p_1}
0<\sum\limits_{p_1|k}\delta_{(p_1)^s,r}<\sum\limits_{p_1|k}\frac{1}{p_1^s}<\left(\sum\limits_{p_1|k}\frac{1}{p_1}\right)^s<\left(\frac{\epsilon}{2}\right)^s,
\end{align}
\begin{align}\label{p_2}
0<\sum\limits_{p_1p_2|k}\delta_{(p_1p_2)^s,r}<\sum\limits_{p_1p_2|k}\left(\frac{1}{p_1p_2}\right)^s<\left(\sum\limits_{p_1|k}\frac{1}{p_1^2}\right)^s<\left(\frac{\epsilon}{2}\right)^{2s}
\end{align}
and
\begin{align}\label{p_3}
0<\sum\limits_{p_1p_2p_3|k}\delta_{(p_1p_2p_3)^s,r}<\sum\limits_{p_1p_2p_3|k}\left(\frac{1}{p_1p_2p_3}\right)^s<\left(\sum\limits_{p_1|k}\frac{1}{p_1^3}\right)^s<\left(\frac{\epsilon}{2}\right)^{3s}.
\end{align}
Similar inequalities hold for all the other sums appearing on the right side of (\ref{mudelta}). Combining (\ref{mudelta})-(\ref{p_3}), we get
\begin{align}\label{mudeltaepsilon}
\left|-\sum\limits_{\substack{d|k\\d>1}}\mu(d)\delta_{d^s,r}\right|<\sum\limits_{j=1}^{\infty}\left(\frac{\epsilon}{2}\right)^{sj}=\frac{\epsilon^s}{2^s-\epsilon^s}<\epsilon^s<\epsilon
\end{align}
for any given $\epsilon$ with $0<\epsilon<1$ and $k \in B_\epsilon$. Hence, from (\ref{c_k^s(j)delta}) and (\ref{mudeltaepsilon}) we have
\begin{align*}
\left|\frac{1}{k^{s(r+1)}}\sum \limits_{j=1}^{k^s}j^rc_k^{(s)}(j)
-\left(\frac{J_s(k)}{k^s}-\frac{1}{r+1}\right)\right|<\epsilon
\end{align*}
for any $k \in B_\epsilon$ with $k\geq2$ and $r\geq2$.
\end{proof}
Next we prove that for large $n$, the weighted average $\frac{1}{k_n^{s(r+1)}}\sum\limits_{j=1}^{k_n^s}j^r c_{k_n}^{(s)}(j) $ is very small for some special sequence $\left\lbrace k_n \right\rbrace$.
\begin{theo}
Assume that $\left\lbrace k_n \right\rbrace$ is a sequence of positive integers such that $\omega(k_n)$ is uniformly bounded and all primes dividing $k_n$ tend to infinity as $n\rightarrow \infty$. Then, for any $r\geq1$, we have
\begin{align*}
\lim\limits_{n\rightarrow \infty}\left(\frac{1}{k_n^{s(r+1)}}\sum\limits_{j=1}^{k_n^s}j^r c_{k_n}^{(s)}(j) \right)=1-\frac{1}{r+1}.
\end{align*}
For every $\lambda>1$, there exists a sequence  $\left\lbrace k_n \right\rbrace$ with $\omega(k_n)\sim(\lambda-1)\frac{n}{\log(n)}$ such that
\begin{align*}
\lim\limits_{n\rightarrow \infty}\left(\frac{1}{k_n^{s(r+1)}}\sum\limits_{j=1}^{k_n^s}j^r c_{k_n}^{(s)}(j) \right)=1-\frac{1}{r+1}
\end{align*}
again holds for any $r\geq1$. Finally, for every $\lambda>1$, there exists a sequence $\left\lbrace k_n \right\rbrace$ with $\omega(k_n)\sim(\lambda-1)\frac{n}{\log(n)}$ such that
\begin{align*}
\lim\limits_{n\rightarrow \infty}\left(\frac{1}{k_n^{s(r+1)}}\sum\limits_{j=1}^{k_n^s}j^r c_{k_n}^{(s)}(j) \right)=\frac{1}{2}-\frac{S_r(2^s)}{2^{s(r+1)}}
\end{align*}
holds for any $r\geq1$.
 \end{theo}
 \begin{proof}
Asume that $\left\lbrace k_n \right\rbrace$ is a sequence of positive integers such that $\omega(k_n)$ is bounded and all primes dividing $k_n $ tend to infinity as $n\rightarrow \infty$. From (\ref{generalized}), we have
\begin{align}\label{productform}
\frac{1}{k_{n}^{s(r+1)}}\sum \limits_{j=1}^{k_{n}^s}j^rc_{k_n}^{(s)}(j)= \frac{J_s(k_n)}{2k_{n}^s}+\frac{1}{r+1}\sum \limits_{m=1}^{[\frac{r}{2}]}\binom{r+1}{2m}B_{2m}\prod \limits_{p|k_n}\left(1-\frac{1}{p^{2ms}}\right)
\end{align}
for $r\geq1$ and $k_n\geq2$, where the sum over $m$ is taken to be zero when $r=1$. By our assumptions, we see that
\begin{align}\label{j_s(k_n)/k_n}
\lim\limits_{n\rightarrow \infty}\frac{J_s(k_n)}{2k_{n}^s}=\frac{1}{2}\lim\limits_{n\rightarrow \infty}\left(\prod \limits_{p|k_n}\left(1-\frac{1}{p^{s}}\right)\right)=\frac{1}{2}
\end{align}
and
\begin{align}\label{1/p^2ms}
\lim\limits_{n\rightarrow \infty}\left(\prod \limits_{p|k_n}\left(1-\frac{1}{p^{2ms}}\right)\right)=1
\end{align}
for any $m\geq 1$. By the recursion (\ref{recursionformula}) satisfied by Bernoulli numbers  we have
 $ \sum\limits_{j=2}^{r}\binom{r+1}{j}B_j=\frac{r+1}{2}-1 $.
Since $B_j=0$ when $j$ is odd and $j\geq 3$, we may rewrite this as
\begin{align}\label{recursion}
\sum \limits_{m=1}^{[\frac{r}{2}]}\binom{r+1}{2m}B_{2m}=\frac{r+1}{2}-1.
\end{align}
From (\ref{productform})-(\ref{recursion}) we get
\begin{align*}
\lim\limits_{n\rightarrow \infty}\left(\frac{1}{k_n^{s(r+1)}}\sum\limits_{j=1}^{k_n^s}j^r c_{k_n}^{(s)}(j) \right)=1-\frac{1}{r+1}.
\end{align*}
Write
\begin{align}\label{k_n}
k_n=\prod \limits_{\substack{n<p\leq \lambda n\\p \text{ prime}}}p
\end{align}
where $n\geq1$ and $\lambda >1$, where empty products are equal to one. Let $\pi(x)$ be the counting function on the set of all primes. Then by the prime number theorem,
\begin{align*}
\omega(k_n)=\pi(\lambda n)-\pi(n)\sim\frac{\lambda n}{\log \lambda n}-\frac{n}{\log n}
\end{align*}
which is asymptotically equal to $(\lambda -1)\frac{n}{\log n}$.
Now by \cite[Equation 5.5]{alkan2012distribution} we have
\begin{align*}
\lim \limits_{n\rightarrow \infty}\prod \limits_{n<p\leq \lambda n}\left(1-\frac{1}{p}\right)=1
\end{align*}
from which it follows that
\begin{align}\label{j_s/k^n}
\lim \limits_{n\rightarrow \infty}\frac{J_s(k_n)}{k_{n}^s}=\lim \limits_{n\rightarrow \infty}\prod \limits_{n<p\leq \lambda n}\left(1-\frac{1}{p^s}\right)=1
\end{align}
and 
\begin{align}\label{1/p^2ms1}
\lim \limits_{n\rightarrow \infty}\prod \limits_{n<p\leq \lambda n}\left(1-\frac{1}{p^{2ms}}\right)=1
\end{align}
for any $m\geq 1$. Hence, from (\ref{productform}), (\ref{j_s/k^n}), (\ref{1/p^2ms1}) and the recursion of Bernoulli numbers, we get
\begin{align*}
\lim\limits_{n\rightarrow \infty}\left(\frac{1}{k_n^{s(r+1)}}\sum\limits_{j=1}^{k_n^s}j^r c_{k_n}^{(s)}(j) \right)=1-\frac{1}{r+1}.
\end{align*}
Next we define 
\begin{align}\label{k_n2}
k_n=2\prod \limits_{\substack{n<p\leq \lambda n\\p \text{ prime}}}p
\end{align}
for $n\geq1$ and $\lambda >1$. From (\ref{c_k^s(j)delta}), we have
\begin{align}\label{ck_n}
\frac{1}{k_n^{s(r+1)}}\sum \limits_{j=1}^{k_n^s}j^rc_{k_n}^{(s)}(j)
=\frac{J_s(k_n)}{k_n^s}-\frac{1}{r+1}-\sum\limits_{\substack{d|k_n\\d>1}}\mu(d)\delta_{d^s,r}
\end{align}
for $r\geq 2$ and $k_n\geq2$.
Now we recall the Merten's theorem (\cite[Theorem 429]{hardy1979introduction}) which states that $\prod \limits_{p\leq x}\left(1-\frac{1}{p}\right)\sim\frac{e^{-\gamma}}{\log x}$, where $\gamma$ is \textit{Euler's constant}(also known as \textit{Euler-Mascheroni constant}).

From (\ref{k_n2}) and Mertens' theorem, we have
\begin{align}\label{j_s1/2}
\lim \limits_{n\rightarrow \infty}\frac{J_s(k_n)}{k_{n}^s}=\lim \limits_{n\rightarrow \infty}\left[\left(1-\frac{1}{2^s}\right)\prod \limits_{\substack{n<p\leq \lambda n\\p \text{ prime}}}\left(1-\frac{1}{p^s}\right)\right]=\frac{2^s-1}{2^s}.
\end{align}
Note that
\begin{align}\label{deltas_r}
\delta_{2^s, r}=\frac{1}{r+1}-\frac{S_r(2^s)}{2^{s(r+1)}}
\end{align}
and
\begin{align}\label{deltad>2}
\sum\limits_{\substack{d|k_n\\d>1}}\mu(d)\delta_{d^s,r}=-\delta_{2^s, r}+\sum\limits_{\substack{d|k_n\\d>2}}\mu(d)\delta_{d^s,r}.
\end{align}
Now
\begin{align}\label{muevenodd}
\sum\limits_{\substack{d|k_n\\d>2}}\mu(d)\delta_{d^s,r}=\sum\limits_{\substack{d|k_n\\d>2\\ d\text{ odd}}}\mu(d)\delta_{d^s,r}+\sum\limits_{\substack{d|k_n\\d>2\\ d\text{ even}}}\mu(d)\delta_{d^s,r}.
\end{align}
From \cite[Equation 5.13]{alkan2012distribution}, we get
\begin{align}\label{1/podd}
\sum \limits_{\substack{p|k_n\\ p \text{ odd}}}\frac{1}{p^s}\leq\left(\sum \limits_{\substack{p|k_n\\ p \text{ odd}}}\frac{1}{p}\right)^s= \left(O(1)\right)^s=O(1)
\end{align}
as $n$ tends to infinity. Using a similar argument as in the proof of Theorem \ref{theorem 1}, it follows from (\ref{1/podd}) that
\begin{align}\label{muodd}
\left|\sum\limits_{\substack{d|k_n\\d>2\\ d\text{ odd}}}\mu(d) \delta_{d^s,r} \right|=O(1)
\end{align}
as $n$ tends to infinity. Also, we have
\begin{align}\label{mueven2m}
\sum\limits_{\substack{d|k_n\\d>2\\ d\text{ even}}}\mu(d)\delta_{d^s,r}=\sum\limits_{\substack{m|\frac{k_n}{2}\\m>1}}\mu(2m)\delta_{(2m)^s,r}=-\sum\limits_{\substack{m|\frac{k_n}{2}\\m>1}}\mu(m)\delta_{(2m)^s,r}.
\end{align}
Since $0<\delta_{(2m)^s,r}<\frac{1}{(2m)^s}$, from (\ref{mueven2m}) and similar arguments as above, we get
\begin{align}\label{mueven}
\left|\sum\limits_{\substack{d|k_n\\d>2\\ d\text{ even}}}\mu(d) \delta_{d,r} \right|=O(1)
\end{align}
as $n$ tends to infinity. Hence, from  (\ref{muevenodd}), (\ref{muodd}), (\ref{mueven}), we get
\begin{align}\label{mudeltao(1)}
\left|\sum\limits_{\substack{d|k_n\\d>2}}\mu(d) \delta_{d^s,r} \right|=O(1)
\end{align}
as $n$ tends to infinity. Now combining (\ref{ck_n})-(\ref{deltad>2}) and (\ref{mudeltao(1)}), it follows that
\begin{align*}
\lim\limits_{n\rightarrow \infty}\left(\frac{1}{k_n^{s(r+1)}}\sum\limits_{j=1}^{k_n^s}j^r c_{k_n}^{(s)}(j) \right)=\frac{2^s-1}{2^s}-\frac{S_r(2^s)}{2^{s(r+1)}}.
\end{align*}
for $r\geq 2$. When $r=1$, from Equation(\ref{productform}) we get the left-hand side of the idenity as $\lim\limits_{n\rightarrow \infty}\left(\frac{1}{k_n^{2s}}\sum\limits_{j=1}^{k_n^s}j c_{k_n}^{(s)}(j) \right)= \lim \limits_{n\rightarrow \infty}\frac{J_s(k_n)}{2k_{n}^s}=\frac{2^s-1}{2^{s+1}}$ and the right-hand side of the identity as $\frac{2^s-1}{2^s}-\frac{S_1(2^s)}{2^{2s}}=\frac{2^s-1}{2^s}-\frac{1}{2^{2s}}\left(\sum \limits_{j=1}^{2^s-1}j\right)=\frac{2^s-1}{2^s}-\frac{2^s(2^s-1)}{2^{2s+1}}=\frac{2^s-1}{2^{s+1}}$. This completes the proof in all the cases.
 \end{proof}
 Now we prove that the weighted average remains positive.
 \begin{theo}
 The average value over k of \begin{align*}
 \frac{1}{k^{s(r+1)}}\sum \limits_{j=1}^{k^s}j^rc_k^{(s)}(j)
 \end{align*} is positive for all $r\geq1$.
 \end{theo} 
 \begin{proof}
  Proceeding as in the proof of Theorem \ref{theorem 1},
we use (\ref{c_k^s(j)delta}) to get
\begin{align}\label{c_k^s(j)mudelta}
\sum\limits_{k\leq x}\left(\frac{1}{k^{s(r+1)}}\sum \limits_{j=1}^{k^s}j^rc_k^{(s)}(j)\right)
&=1+\sum\limits_{2\leq k \leq x}\left(\frac{1}{k^{s(r+1)}}\sum \limits_{j=1}^{k^s}j^rc_k^{(s)}(j)\right)\\
&=\sum\limits_{k\leq x}\frac{J_s(k)}{k^s}-\frac{1}{r+1}\sum\limits_{2\leq k \leq x}1-\sum\limits_{k\leq x} \sum\limits_{\substack{d|k\\d>1}}\mu(d)\delta_{d^s,r}.\nonumber
\end{align} 
From \cite[Theorem 6.10]{mccarthy2012introduction}, we have $
\sum\limits_{k\leq x} \frac{J_s(k)}{k^s}=\frac{x}{\zeta(s+1)}+O(1)$. Therefore
\begin{align}\label{c_k^s(j)O(1)}
&\sum\limits_{k\leq x}\left(\frac{1}{k^{s(r+1)}}\sum \limits_{j=1}^{k^s}j^rc_k^{(s)}(j)\right)= \left(\frac{1}{\zeta(s+1)}-\frac{1}{r+1} \right)x-\sum\limits_{k\leq x} \sum\limits_{\substack{d|k\\d>1}}\mu(d)\delta_{d^s,r}+O(1).
\end{align}
Now consider
\begin{align}\label{c_k^s(j)O}
&\sum\limits_{k\leq x} \sum\limits_{\substack{d|k\\d>1}}\mu(d)\delta_{d^s,r}\\&=\sum\limits_{1<d\leq x} \sum\limits_{q\leq\frac{x}{d}}\mu(d)\delta_{d^s,r}\nonumber\\
&=\sum\limits_{1<d\leq x} \mu(d)\delta_{d^s,r}\left[\frac{x}{d}\right]\nonumber\\
&=\sum\limits_{1<d\leq x} \mu(d)\delta_{d^s,r}\left(\frac{x}{d}+O(1)\right)\nonumber\\
&=\sum\limits_{1< d\leq x} \frac{\mu(d)}{d}x\delta_{d^s,r}+O\left(\sum\limits_{1<d\leq x}\frac{1}{d^s}\right) \text{ since }0<\delta_{d^s,r}<\frac{1}{d^s}\nonumber\\
&=x\sum \limits_{d=2}^{\infty}\frac{\mu(d)}{d}\delta_{d^s, r}-x\sum\limits_{d>x}\frac{\mu(d)}{d}\delta_{d^s, r}+O\left(\frac{x^{1-s}}{1-s}+\zeta(s)+O(x^{-s})\right) \text{ \cite[Theorem 3.2]{tom1976introduction}}\nonumber\\
&=x\sum \limits_{d=2}^{\infty}\frac{\mu(d)}{d}\delta_{d^s, r}-xO(x^{-s})+O\left(\frac{x^{1-s}}{1-s}+\zeta(s)+O(x^{-s})\right)\nonumber\\
&=x\sum \limits_{d=2}^{\infty}\frac{\mu(d)}{d}\delta_{d^s, r}+O(1).\nonumber
\end{align}
From (\ref{c_k^s(j)O(1)}) and (\ref{c_k^s(j)O}), the desired average value over $k$ is 
\begin{align}\label{average}
&\frac{1}{\zeta(s+1)}-\frac{1}{r+1}-\sum \limits_{d=2}^{\infty}\frac{\mu(d)}{d}\delta_{d^s, r}\\
&=\frac{1}{\zeta(s+1)}-\frac{1}{r+1}+\frac{\delta_{2, r}}{2}+\frac{\delta_{3, r}}{3}+\frac{\delta_{5, r}}{5}+\ldots+\frac{\delta_{n, r}}{n}-\sum \limits_{d=n+1}^{\infty}\frac{\mu(d)}{d}\delta_{d^s, r}\nonumber\\
&>\frac{1}{\zeta(s+1)}-\frac{1}{r+1}-\sum \limits_{d=n+1}^{\infty}\frac{\mu(d)}{d}\delta_{d^s, r}.\nonumber
\end{align}
Now
\begin{align}\label{averageintegral}
\left|\sum \limits_{d=n+1}^{\infty}\frac{\mu(d)}{d}\delta_{d^s, r}\right|\leq \sum \limits_{d=n+1}^{\infty}\frac{1}{d^{s+1}}\leq \int \limits_{n}^{\infty}\frac{1}{t^{s+1}}\,dt=\frac{1}{n^{s}}.
\end{align}
Combining (\ref{average}) and (\ref{averageintegral}) we get
\begin{align*}
\frac{1}{\zeta(s+1)}-\frac{1}{r+1}-\sum \limits_{d=2}^{\infty}\frac{\mu(d)}{d}\delta_{d^s, r}\geq \frac{1}{\zeta(s+1)}-\frac{1}{r+1}-\frac{1}{n^s}>0,
\end{align*}
 for some positive integer $n$.
Thus the average value over $k$ is positive for all $r\geq1$. Finally when $k\geq 2$ is fixed, using (\ref{c_k^s(j)delta}) and the fact that for any fixed $d>1$,
\begin{align*}
\lim \limits_{r\rightarrow \infty}\delta_{d,r}=\lim \limits_{r\rightarrow \infty}\left(\frac{1}{r+1}-\frac{S_r(d^s)}{d^{s(r+1)}}\right)=0,
\end{align*}
we can see that
\begin{align*}
\lim \limits_{r\rightarrow \infty}\left(\frac{1}{k^{s(r+1)}}\sum \limits_{j=1}^{k^s}j^rc_k^{(s)}(j)\right)=\frac{J_s(k)}{k^s}>0
\end{align*}
which is also true when $k=1$.
 \end{proof}
Now we give a lower bound for the maximum value of sum of generalized Ramanujan sums.
\begin{theo}\label{3.4}
For any $k\geq1$, the estimate \begin{align*}
\max\limits_{N}\left|\sum \limits_{j=1}^{N^s}c_k^{(s)}(j) \right|\geq \frac{J_{2s}(k)}{4k^s}+\frac{J_s(k)}{2} 
\end{align*}holds.
\end{theo}
\begin{proof}
We have $c_1^{(s)}(j)=1$.
Hence at $k=1$, the required maximum is infinite and the inequality trivially holds.
Now assume that $k\geq 2$. We have 
\begin{align}\label{c_k^s=0}
\sum \limits_{j=1}^{k^s}c_k^{(s)}(j)=0 \text{ \cite[corollary 2.1]{cohen1955extension}} .
\end{align}
From Equation (\ref{generalized}) with $r=3$, we get
\begin{align*}
\frac{1}{k^{4s}}\sum \limits_{j=1}^{k^s}j^3c_k^{(s)}(j)= \frac{J_s(k)}{2k^s}+\frac{J_{2s}(k)}{4k^{2s}} \text{ since } B_2=\frac{1}{6}.
\end{align*} 
Therefore
\begin{align}\label{j^3c_k^s(j)}
\sum \limits_{j=1}^{k^s}j^3c_k^{(s)}(j)= \left(\frac{J_{2s}(k)}{4k^{s}}+\frac{J_s(k)}{2}\right)k^{3s}.
\end{align}
Using Abel's identity \cite[Theorem 4.2]{tom1976introduction} and (\ref{c_k^s=0}), we get
\begin{align*}
\sum \limits_{j=1}^{k^s}j^3c_k^{(s)}(j)&=-c_k^{(s)}(1)-\int \limits_1^{k^s}3t^2\left(\sum \limits_{1\leq j \leq t}c_k^{(s)}(j)\right)\,dt\\
&=-\mu(k)-\int \limits_1^{k^s}3t^2\left(\sum \limits_{1\leq j \leq t}c_k^{(s)}(j)\right)\,dt.
\end{align*}
Thus
\begin{align}\label{j^3c_k^smax}
\left|\sum \limits_{j=1}^{k^s}j^3c_k^{(s)}(j)\right|&\leq \int \limits_1^{k^s}\max\limits_{N}\left|\sum \limits_{ j= 1}^{N^s}c_k^{(s)}(j)\right|3t^2\,dt\\
&=\max\limits_{N}\left|\sum \limits_{ j= 1}^{N^s}c_k^{(s)}(j)\right|\int \limits_1^{k^s}3t^2\,dt\nonumber\\
&=\max\limits_{N}\left|\sum \limits_{ j= 1}^{N^s}c_k^{(s)}(j)\right|(k^{3s}-1)\nonumber\\
&\leq \left(\max\limits_{N}\left|\sum \limits_{ j= 1}^{N^s}c_k^{(s)}(j)\right|\right)k^{3s}.\nonumber
\end{align}
From (\ref{j^3c_k^s(j)}) and (\ref{j^3c_k^smax}), we get
\begin{align*}
\left(\max\limits_{N}\left|\sum \limits_{ j= 1}^{N^s}c_k^{(s)}(j)\right|\right)k^{3s}\geq \left(\frac{J_{2s}(k)}{4k^{s}}+\frac{J_s(k)}{2}\right)k^{3s}.
\end{align*}
and so
\begin{align*}
\max\limits_{N}\left|\sum \limits_{ j= 1}^{N^s}c_k^{(s)}(j)\right|\geq \frac{J_{2s}(k)}{4k^{s}}+\frac{J_s(k)}{2}.
\end{align*}
\end{proof}
\begin{coro}
For any integer $r\geq 1$, $k\geq 1$ and $s$, the inequality
\begin{align*}
\left|\frac{1}{2}+\left( \frac{1}{r+1}\sum \limits_{m=1}^{[\frac{r}{2}]}\binom{r+1}{2m}B_{2m}\frac{J_{2ms}(k)}{k^{2ms}}\right)\frac{k^s}{J_s(k)}\right|\leq k^{s(s-1)}2^{\omega(k)}
\end{align*}holds.
\end{coro}
\begin{proof}
Using (\ref{recursion}), we get
$\frac{1}{r+1}\sum \limits_{m=1}^{[\frac{r}{2}]}\binom{r+1}{2m}B_{2m}=\frac{1}{2}-\frac{1}{r+1}
$.
Therefore $\left|\frac{1}{2}+\frac{1}{r+1}\sum \limits_{m=1}^{[\frac{r}{2}]}\binom{r+1}{2m}B_{2m}\right|=\left|1-\frac{1}{r+1}\right|\leq 1$ since $r\geq1$. Thus the inequality holds when $k=1$. Now assume that $k\geq 2$.
From (\ref{generalized}), we have
\begin{align}\label{j^rc_k^s(j)}
\sum \limits_{j=1}^{k^s}j^rc_k^{(s)}(j)&=\left( \frac{J_s(k)}{2k^s}+\frac{1}{r+1}\sum \limits_{m=1}^{[\frac{r}{2}]}\binom{r+1}{2m}B_{2m}\frac{J_{2ms}(k)}{k^{2ms}}\right)k^{s(r+1)}\\
&=\left[\frac{1}{2}+\left( \frac{1}{r+1}\sum \limits_{m=1}^{[\frac{r}{2}]}\binom{r+1}{2m}B_{2m}\frac{J_{2ms}(k)}{k^{2ms}}\right)\frac{k^s}{J_s(k)}\right]J_s(k)k^{sr}.\nonumber
\end{align}
Using Abel's identity, as in the proof of Theorem \ref{3.4} we get
\begin{align}\label{j^rc_k(j)max}
\left|\sum \limits_{j=1}^{k^s}j^rc_k^{(s)}(j) \right|\leq \max\limits_{N}\left|\sum \limits_{ j= 1}^{N^s}c_k^{(s)}(j)\right|k^{sr}\leq \sum \limits_{ j= 1}^{k^s}\left|c_k^{(s)}(j)\right|k^{sr}.
\end{align}
Now we consider $\sum \limits_{ j= 1}^{k^s}\left|c_k^{(s)}(j)\right|$. Let $k=\prod\limits_{v=1}^q p_v^{a_v}$ be the prime factorization of $k$ into distinct primes. Recall that we have $c_k^{(s)}(j)=\frac{J_s(k)\mu(d)}{J_s(d)}$, where $d^s=\frac{k^s}{(j, k^s)_s}$. If we write $j=\delta^s\gamma$, where $(\gamma, k^s)_s=1$, then $d^s=\frac{k^s}{\delta^s}$
so that $d=\frac{k}{\delta}$. Therefore, $c_k^{(s)}(j)=0$ unless $\frac{k}{\delta}$ is a square free number. Then $\delta=\prod\limits_{v=1}^q p_v^{b_v}$ with $a_v-1\leq b_v\leq a_v$ for each $v$ and there are exactly $2^{\omega(k)}$ values of $d$ such that $c_k^{(s)}(j)\neq0$.
Thus $|c_k^{(s)}(j)|\leq\frac{J_s(k)}{J_s(\frac{k}{\delta})} $ and so
\begin{align}\label{|c_k^(s)(j)|}
\sum \limits_{ j= 1}^{k^s}\left|c_k^{(s)}(j)\right|\leq \sum \limits_{\substack{1\leq j \leq k^s\\c_k^{(s)}(j)\neq 0}}\left|c_k^{(s)}(j)\right|\leq \sum \limits_{\substack{1\leq j \leq k^s\\c_k^{(s)}(j)\neq 0}}\frac{J_s(k)}{J_s(d)}.
\end{align}
The condition $1\leq j\leq k^s$ and $c_k^{(s)}(j)\neq 0$ is equivalent to $1\leq j\leq \frac{k^s}{\delta^s}$ and $(j, \frac{k^s}{\delta^s})_s=1$. The number of such integers is $J_s\left(\frac{k^s}{\delta^s}\right)$. From  (\ref{|c_k^(s)(j)|}) we get
\begin{align}\label{c_k^s(j)omega}
\sum \limits_{ j= 1}^{k^s}\left|c_k^{(s)}(j)\right|&\leq \sum \limits_{\substack{1\leq j \leq k^s\\c_k^{(s)}(j)\neq 0}}\frac{J_s(k)}{J_s(d)}\\
&=J_s(k)\frac{J_s\left(\frac{k^s}{\delta_1^s}\right)}{J_s\left(\frac{k}{\delta_1}\right)}+J_s(k)\frac{J_s\left(\frac{k^s}{\delta_2^s}\right)}{J_s\left(\frac{k}{\delta_2}\right)}+\cdots + J_s(k)\frac{J_s\left(\frac{k^s}{\delta_{2^{\omega(k)}}^s}\right)}{J_s\left(\frac{k}{\delta_{2^{\omega(k)}}}\right)}\nonumber\\
&=J_s(k)\left[\frac{\frac{k^{s^2}}{\delta_1^{s^2}}\prod\limits_{p|\frac{k^s}{\delta_1^s}}\left(1-\frac{1}{p^s} \right)}{\frac{k^{s}}{\delta_1^{s}}\prod\limits_{p|\frac{k}{\delta_1}}\left(1-\frac{1}{p^s} \right)}+\cdots+\frac{\frac{k^{s^2}}{\delta_{2^{\omega(k)}}^{s^2}}\prod\limits_{p|\frac{k^s}{\delta_{2^{\omega(k)}}^s}}\left(1-\frac{1}{p^s} \right)}{\frac{k^{s}}{\delta_{2^{\omega(k)}}^{s}}\prod\limits_{p|\frac{k}{\delta_{2^{\omega(k)}}}}\left(1-\frac{1}{p^s} \right)} \right]\nonumber\\
&=J_s(k)k^{s^2-s}\left[\frac{1}{\delta_1^{s^2-s}}+\cdots+\frac{1}{\delta_{2^{\omega(k)}}^{s^2-s}}\right]\nonumber\\
&\leq J_s(k)k^{s^2-s}(1+\cdots+1)\nonumber\\
&\leq J_s(k)k^{s^2-s}\,2^{\omega(k)}.\nonumber
\end{align}
From (\ref{j^rc_k^s(j)}), (\ref{j^rc_k(j)max}) and (\ref{c_k^s(j)omega}), we get
\begin{align*}
\left|\frac{1}{2}+\left( \frac{1}{r+1}\sum \limits_{m=1}^{[\frac{r}{2}]}\binom{r+1}{2m}B_{2m}\frac{J_{2ms}(k)}{k^{2ms}}\right)\frac{k^s}{J_s(k)}\right|J_s(k)k^{sr}
&\leq \sum \limits_{ j= 1}^{k^s}\left|c_k^{(s)}(j)\right|k^{sr}\\
 &\leq J_s(k)k^{s(s-1)}2^{\omega(k)}k^{sr}.
\end{align*}
Thus as we claimed,
\begin{align*}
\left|\frac{1}{2}+\left( \frac{1}{r+1}\sum \limits_{m=1}^{[\frac{r}{2}]}\binom{r+1}{2m}B_{2m}\frac{J_{2ms}(k)}{k^{2ms}}\right)\frac{k^s}{J_s(k)}\right|\leq k^{s(s-1)}2^{\omega(k)}.
\end{align*}
\end{proof}

\section{Acknowledgements}
 The first author thanks the Kerala State Council for Science,Technology and Environment, Thiruvananthapuram, Kerala, India for providing financial support for carrying out this research work.

\bibliographystyle{plain}


\end{document}